\documentclass[12pt,a4paper, reqno]{amsart}
\usepackage[utf8]{inputenc}

\usepackage[pdfborder={0 0 0}]{hyperref}

\usepackage{latexsym}

\calclayout
 
\usepackage{graphicx} 
\usepackage{amsmath}  
\usepackage{amssymb}  
\usepackage{amsthm} 
\usepackage{xcolor}

\usepackage{amsfonts} 
\usepackage{mathtools} 
\usepackage{physics}  
\usepackage{bm}       
\usepackage{cancel}   
\usepackage{bbm}      
\usepackage{nicefrac} 
\usepackage{siunitx}  
\usepackage{xfrac}    

\theoremstyle{plain}
\newtheorem{theorem}{Theorem}

\newtheorem{corollary}[theorem]{Corollary}

\newtheorem{definition}[theorem]{Definition}

\numberwithin{theorem}{section}
\numberwithin{equation}{section}
\allowdisplaybreaks

\DeclareMathOperator{\sgn}{sgn}
\def\eps{\varepsilon}
\def\pa{\partial}

\DeclareMathOperator{\Div }{div}

\def\pa{\partial}
\def\cal{\mathcal}
\let\mib=\boldsymbol

\def\R{{\bf R}}
\def\N{{\bf N}}

\def\eps{\varepsilon}

\def\mx{{\bf x}}
\def\mxi{{\mib \xi}}

\def\mff{{\mathfrak f}}

\def\mff{{\mathfrak f}}
\def\mx{{\bf x}}

\def\pa{\partial}
\def\cal{\mathcal}
\let\mib=\boldsymbol

\def\R{\Bbb{R}}

\def\N{\Bbb{N}}

\def\eps{\varepsilon}
\def\mxi{{\mib \xi}}
\def\mx{{\mib x}}

\def\sgn{\operatorname{sgn}}

\begin{document}

\title[Well-posedness -- discontinuous flux]
{Well posedness for multidimensional scalar conservation laws with discontinuous
 non-degenerate flux}

\author[D.Mitrovic]{D. Mitrovic}
\address{Darko Mitrovic, University of Vienna, 
Faculty of Mathematics, Oskar Morgenstern Platz-1,
1090 Vienna, Austria}
\address{and University of Montenegro, 
Faculty of Science and Mathematics, Cetinjski put bb,
81000 Podgorica, Montenegro}
\email{darko.mitrovic@univie.ac.at}

\subjclass[2020]{35L65; 35L45}

\keywords{Conservation law, discontinuous flux, vanishing viscosity; entropy 
conditions}

\begin{abstract}
We show well posedness of the Cauchy problem for scalar conservation laws with 
discontinuous non-degenerate flux which locally have the form
\begin{equation*}
\begin{split}
&\partial_t u + \sum\limits_{k=1}^{d}\partial_{x_k}\left(f^k_L(\mx,u) 
H(-x_j+\zeta(\hat{x}_j)) + f^k_R(\mx,u) H(x_j-\zeta(\hat{x}_j))\right) = 0,\\
& u\big|_{t=0}=u_0, \ \ a \leq u_0 \leq b,  \ \ a,b\in \R, 
\end{split}
\end{equation*} where $\mx\in U\subset \R^d$ and $\hat{x}_j=(x_j,\dots, x_{j-1},x_{j+1}, \dots,x_d)\in \R^{d-1}$ for some open set $U$ and the Heaviside function $H$. The flux functions are such that 
$$
f^k_L, f^k_R\in C^1(\R^{d+1}))  \ \ \text{ for $k=1,\dots, d,$  and }  \ \ f^k_{L,R}(\mx,a)=f^k_{L,R}(\mx,b) =0
$$ while the interface function satisfies $\zeta\in C^1(\R^{d-1})$. 

We also consider the problem of the general form 
\begin{equation*}
\begin{split}
&\pa_t u+\Div \mff(\mx,u)=0, \ \ \mff \in BV(\R^d;C^1(\R)), \ \ 
\mff(\mx,a)=\mff(\mx,b)=0,\\
& u\big|_{t=0}=u_0, \ \ a \leq u_0 \leq b,  \ \ a,b\in \R,
\end{split}
\end{equation*}  and construct a stable semigroup of solutions. 
\end{abstract}

\date{\today}

\maketitle

\section{Introduction}

We consider a multidimensional scalar conservation law with a spatially discontinuous flux, written here in the general form
\begin{equation}\label{eqmain}
\partial_t u + \Div \mff(\mx,u) = 0, \qquad u|_{t=0}=u_0(\mx), \quad \mx\in\R^d, \; t>0,
\end{equation}
where $u(t,\mx)$ is the unknown scalar function, and the flux $\mff:\R^d\times \R \to \R^d$ may have discontinuities in the spatial variable to be specified below. We assume $a\le u_0(\mx)\le b$ for some constants $a,b\in\R$, and that $\mff(\mx,\lambda)$ is of class $C^1$ in the $\lambda$-variable and has bounded variation in the $\mx$-variable (more precisely, $\mff\in BV(\R^d;C^1(\R))^d$). 

Conservation laws of the form \eqref{eqmain} model various physical phenomena, such as traffic flow on networks with variable road conditions, two-phase flow in heterogeneous porous media, sedimentation processes, and so on. When the flux function is spatially homogeneous (or sufficiently smooth in $\mx$), the standard theory of Kru\v{z}kov \cite{Kru} provides a well-posedness framework for the Cauchy problem: there is a unique entropy solution $u(t,\mx)$ satisfying the entropy inequalities and depending continuously on the initial data in $L^1$. However, if $\mff$ has discontinuities in $\mx$, the classical Kru\v{z}kov entropy condition is not enough to single out a unique solution. Indeed, in one space dimension it has been shown that many different admissible solution concepts can be defined for the same discontinuous-flux problem, depending on additional conditions imposed at the points of flux discontinuity (see, e.g., \cite{AKR, crasta1} and references therein). The choice of the “correct” admissibility criterion is often guided by the physical modeling context. One popular selection principle is the \emph{vanishing viscosity} criterion, wherein one declares the physically relevant solution to be the limit of solutions of \eqref{eqmain} with a small viscous term $\eps \Delta u$ added. This approach has the advantage of being tied to a specific modeling assumption (negligible but nonzero diffusion) and has been extensively studied (see \cite{AKR} for the one-dimensional case with a single discontinuity where also a general theory of $L^1$-dissipative solvers was considered).

In multiple space dimensions, the analytical difficulties are even more pronounced. The flux discontinuities may occur along irregular hypersurfaces in $\R^d$, and there is no canonical “Riemann problem” that isolates the behavior at a single interface as in the one-dimensional case. A general existence result for entropy solutions was proved by E.~Yu.~Panov \cite{panov_arma} using the vanishing viscosity method and compactness arguments, under quite broad assumptions on $\mff$. However, without further admissibility conditions, uniqueness of solutions in the multidimensional discontinuous-flux setting is not guaranteed. Recently, several works have contributed to establishing well-posedness under additional structural assumptions. In \cite{AKR}, the authors introduced the concept of an $L^1$-\emph{dissipative germ} to characterize interface entropy conditions in one space dimension (allowing a finite number of flux discontinuities on the real line). In \cite{AM}, we extended some of these ideas to multidimensional laws with flux discontinuous across a finite number of smooth hypersurfaces. We proposed multidimensional entropy conditions (generalizing the one-dimensional “crossing condition” of \cite{KRT}) and proved that, for fluxes with finitely many spatial discontinuity surfaces, the vanishing viscosity limit does indeed converge to a unique entropy solution satisfying those conditions (thus yielding well-posedness). On the other hand, \cite{crasta1} established a uniqueness theory for entropy solutions in several dimensions by developing a kinetic formulation and imposing an adapted entropy condition at flux discontinuities. In a subsequent work, \cite{crasta2}, the authors further analyzed the structure of entropy solutions, proving that any entropy solution (under a genuine nonlinearity hypothesis similar to \eqref{non-deg}) admits strong traces on the flux discontinuity set and satisfies a generalized Kato inequality. These properties were used to deduce uniqueness of solutions, given an appropriate condition relating the traces of any two solutions at the discontinuities.

Despite these advances, a complete well-posedness theory for multidimensional conservation laws with general discontinuous fluxes has remained elusive. In particular, the question of existence of a solution that satisfies a given entropy admissibility criterion (such as the multidimensional entropy inequality of \cite{AM}) was open in full generality, since compactness methods alone do not ensure that the limit of vanishing viscosity approximations satisfies the desired entropy condition. 

In the current contribution, we make a step forward by proving both existence and uniqueness to multidimensional scalar conservation laws with flux discontinuity along a non-intersecting manifolds (see Figure 1). More precisely, beside already introduced conditions on flux, we assume that $\mff$ satisfies the boundary conditions 
\begin{equation}\label{bounded}
\mff(\mx,a)=\mff(\mx,b)=0,
\end{equation} which provide the maximum principle. Moreover, we assume the non-degeneracy condition: for almost every $\mx\in\R^d$ and every direction $\mxi\in S^{d-1}$ (the unit sphere in $\R^d$), the function 
\begin{equation}\label{non-deg}
\lambda \mapsto \mxi\cdot \partial_\lambda \mff(\mx,\lambda)
\end{equation}
is not identically zero on any non-degenerate interval. 

Let us now precise geometrical conditions that the flux should satisfy. Roughly speaking, we assume that the flux discontinuities are confined to a finite number of smooth “interfaces” that do not intersect each other. Specifically, we assume:
\begin{itemize}
\item[(a)] There exist sets $Q_j\subset \R^d$ for $j=1,\dots,d$ such that, for each $j$ and each point $\mx_0\in \R^d\setminus Q_j$, there is a neighborhood $B(\mx_0,r)$ in which the flux $\mff$ can be written in the form
\begin{equation}\label{local-flux}
\begin{split}
\Div \mff(\mx,\lambda) &= \sum_{k=1}^d \partial_{x_k}\Big( f^k_L(\mx,\lambda)\,H(-x_j + \zeta_{j}(\hat{x}_j)) \\
&\qquad\qquad\qquad +\; f^k_R(\mx,\lambda)\,H(x_j - \zeta_{j}(\hat{x}_j))\Big),
\end{split}
\end{equation}
for some $C^2$ functions $f^k_L$ and $f^R_L$, $k=1,\dots,d$, and $C^1$-function $\zeta_{j}:\R^{d-1}\to\R$ defining a smooth hypersurface. Here $H$ is the Heaviside step function, and we write $\hat{x}_j=(x_1,\dots,x_{j-1},x_{j+1},\dots,x_d)\in\R^{d-1}$. (In other words, in the local ball $B(\mx_0,r)$, the flux $\mff$ has a single discontinuity surface given by $\{x_j=\zeta_j(\hat{x}_j)\}$, with $f_L$ and $f_R$ the fluxes on each side of this surface.)
\item[(b)] There exists $\epsilon>0$ such that $\cap_{k=1}^d Q _k^{\epsilon}=\emptyset$ for all $i\neq j$. (Here $Q_j^\epsilon := \{\mx:\, {\rm dist}(\mx,Q_j)<\epsilon\}$ denotes the $\epsilon$-neighborhood of $Q_j$.)
\item[(c)] If two neighborhoods $B(\mx_i,r_i)$ and $B(\mx_j,r_j)$ both contain a portion of the same discontinuity interface $Q_k$, then the representations of that interface in $B(\mx_i,r_i)$ and $B(\mx_j,r_j)$ coincide on $B(\mx_i,r_i)\cap B(\mx_j,r_j)$.
\end{itemize}

\begin{figure}[h]
\label{fig:discont}    \centering
    \includegraphics[width=0.6\textwidth]{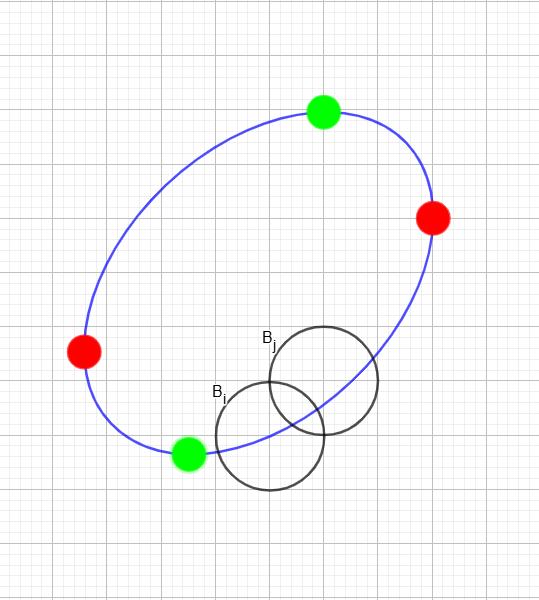}
    \caption{The discontinuity line is the blue ellipse. $B_i$ and $B_j$ are balls in which we have the representation \eqref{local-flux}. Green balls are the set $Q_1^\epsilon$ i.e. they turn into $Q_1$ as $\epsilon\to 0$. Red balls are the set $Q_2^\epsilon$ i.e. they turn into $Q_2$ as $\epsilon\to 0$. }
    \label{fig:jmf}
\end{figure}

A typical configuration satisfying the above requirements is depicted in Figure~\ref{fig:discont}. The blue ellipse represents a discontinuity manifold (in this two-dimensional illustration), while the green circles indicate the set $Q_1^\epsilon$ (which tends to $Q_1$ as $\epsilon \to 0$). The red circles indicate the set $Q_2^\epsilon$ (tending to $Q_2$ as $\epsilon \to 0$). The balls $B_i=B(\mx_i,r_i)$ and $B_j=B(\mx_j,r_j)$ are neighborhoods where the flux can be written in the form \eqref{local-flux}. Their overlap contains a portion of the interface, and assumption (c) ensures that this portion is described by the same function $\zeta$ in both $B_i$ and $B_j$. 

Under assumptions (a)--(c), we are able to show that a kind of vanishing viscosity approximation satisfies the multidimensional entropy admissibility condition derived in \cite{AM} (see inequality \eqref{entropy-ineq} below). Consequently, by the uniqueness theory of \cite{AM, crasta2}, this solution is the unique entropy admissible solution of the Cauchy problem \eqref{eqmain} (for the given flux). In outline, the proof proceeds as follows. For each local region $B(\mx_l,r_l)$ where the flux has the form \eqref{local-flux}, we perform a smooth change of coordinates that “flattens” the discontinuity surface (taking it to a flat interface $\{\tilde{x}_j=0\}$). We then {\em radially extend} the flux so that the Lipschitzity and the flat interface are preserved along entire space. We note that the new idea in this kind of problem is the radial extension of the flux which enables us precise local control of the entropy solutions. 

We then consider a vanishing viscosity approximation for this transformed equation, using a smoothed Heaviside function to regularize the flux discontinuity. Thanks to the entropy stability result of \cite{AM}, the limit as $\eps\to0$ of these local solutions satisfies the entropy inequality \eqref{entropy-ineq} in the region. From this local entropy condition one deduces (again following \cite{AM}) that the solution is $L^1$-unique in a neighborhood (forward light-cone) of the interface segment under consideration. Covering $\R^d$ by finitely many such neighborhoods (using (b) to avoid interference between different interfaces) and patching together the local uniqueness results, we conclude that the limit solution is globally unique. By a standard contraction argument, it is also $L^1$-stable, hence it coincides with the unique entropy solution of \eqref{eqmain} satisfying \eqref{entropy-ineq} on the entire time axis.

In the remainder of the paper, we consider the scalar conservation law with heterogeneous flux~\eqref{eqmain}, approximated by the standard vanishing viscosity method with smoothed flux. It is known \cite{panov_arma} that such an approximation converges, along a subsequence, to a weak solution of~\eqref{eqmain}. The main difficulty lies in the uniqueness of this limit: it is not known whether the limit is unique, or whether different subsequences may converge to distinct weak solutions of~\eqref{eqmain}. While we cannot resolve this uniqueness question, we construct the so-called $(\eps_k)$-vanishing viscosity germ ${\mathcal G}_{(\eps_k)}$, defined as the set of all limit functions obtained (along subsequences $\eps_k \to 0$) from the viscous approximations of~\eqref{eqmain}.  

Invoking the non-degeneracy condition~\eqref{non-deg} together with the compactness results of Panov~\cite{panov_arma}, we prove that there exists at least one subsequence $(\eps_k)$ such that ${\mathcal G}_{(\eps_k)}$ contains a weak solution for \emph{every} initial datum $u_0 \in L^1 \cap L^\infty$ with $a \leq u_0 \leq b$. In other words, for a suitably chosen vanishing sequence $\eps_k \downarrow 0$, the corresponding limit defines a \emph{complete} germ of solutions: for every admissible initial state $u_0$, the vanishing viscosity approximation converges along the sequence $(\eps_k)$. Moreover, exploiting the $L^1$-contraction property of the parabolic approximations, we show that this germ is $L^1$-stable (contractive) with respect to the initial data. By the very construction of ${\mathcal G}_{(\eps_k)}$, there is a uniquely selected solution associated with each initial state---uniqueness holds \emph{within} the chosen germ.  

We note that the notion of germ was introduced in~\cite{AKR}, where it was determined by a set of admissible connections across a discontinuity interface. Here, our notion is somewhat different: we call the germ the set of functions generated by a specific viscosity approximation together with a countable set of initial data (which suffices to generate all other solutions and thus qualifies as a genuine germ).

The paper is organized as follows. After the introduction, in Section~2 below, we deal with the structured discontinuity case: assuming (a)--(c) we verify the entropy inequality~\eqref{entropy-ineq} and conclude the uniqueness (and continuous dependence) of entropy-admissible solutions. Section~3 contains the construction of the $(\eps_k)$-vanishing viscosity germ and the proof of existence and $L^1$-stability of solutions.

We shall finish the introduction by a short historical overview of the theory of scalar conservation laws with discontinuous flux.

\subsection*{Historical overview}

The theory of scalar conservation laws with \emph{continuous} flux rests on the foundational BV and entropy frameworks of Vol'pert and Kru\v{z}kov: BV calculus and traces \cite{Volpert1967} and the $L^1$ well--posedness theory for bounded entropy solutions \cite{Kru}. The multidimensional kinetic formulation from \cite{LPT94} provided a powerful linearization and compactness device and, in particular, a route to averaging lemmas and stability.

When the flux becomes \emph{spatially discontinuous}, classical Kru\v{z}kov entropies are no longer sufficient to enforce uniqueness because they do not encode interface transmission conditions. Early one--dimensional analyses already revealed this phenomenon in oil–reservoir and sedimentation models, beginning with the works \cite{GimseRisebro1992, Diehl1995, DiehlCMP1996}  on Cauchy and Riemann problems with discontinuous flux and on point sources and discontinuous flux. 

Two broad (and complementary) approaches emerged to recover well--posedness.

\smallskip
\emph{(i) Interface entropy/germ viewpoints.}
In one dimension, \cite{AudussePerthame2005} introduced \emph{adapted entropies} tailored to the given discontinuity and proved uniqueness without BV or trace assumptions (see also \cite{KM-MH, pan-jhde-ap, picolli}). A parallel and influential line due to \cite{AdimurthiJNA2004, AdimurthiJHDE2005, AdimurthiNHM2007} developed \emph{optimal entropy solutions} and robust Godunov-type approximations for discontinuous fluxes, yielding existence, stability and convergence for a wide class of connections . We note that some situations analyzed there generate alternative singular $\delta$-type solutions \cite{KMN-JDE, KN, ajlan}.  In \cite{GNPT2007}, the authors recast the problem in terms of \emph{Riemann solvers} and characterized those producing existence and uniqueness. These works crystallized the idea that a selection rule at each interface (an “interface germ”) underlies well--posedness.

In \cite{AKR}, the authors axiomatized this selection via \emph{$L^1$--dissipative germs} and obtained a general one--dimensional well--posedness theory encompassing vanishing viscosity and other regularizations . For multidimensional fluxes that jump across finitely many smooth hypersurfaces, we formulated entropy conditions adapted to the geometry and proved well--posedness by localization and patching via a Kato inequality \cite{AM}.

\smallskip
\emph{(ii) Kinetic/trace formulations.}
A second strand, initiated in the smooth-flux setting by \cite{LPT94}, evolved to handle discontinuous fluxes through kinetic formulations and fine trace theory on the flux jump set. The results from \cite{crasta1,crasta2} established existence of strong traces for entropy solutions on the rectifiable discontinuity set and a \emph{generalized Kato inequality} with a measure concentrated on that set, leading to uniqueness once an interface dissipation rule is prescribed. For additional related existence/uniqueness mechanisms with discontinuous coefficients, see also \cite{BachmannVovelle2006, KlingenbergRisebro1995}.

\smallskip
\emph{Existence via compactness.}
On the existence side, a strong precompactness of vanishing-viscosity approximations for fluxes that are $BV$ in $x$ and $C^1$ in $u$, under a genuine nonlinearity assumption; this yields entropy solutions for very general discontinuous fluxes was proved in \cite{panov_arma}. Combining the mentioned compactness result with either interface germs/entropies \cite{AdimurthiJHDE2005,AKR,AM,AudussePerthame2005} or the kinetic/trace approach \cite{crasta1,crasta2} furnishes a robust strategy to well--posedness: construct limits by vanishing viscosity and prove that the limits satisfy an interface entropy inequality strong enough to recover Kato contraction. The parabolic stability machinery in convection–diffusion with discontinuous coefficients \cite{KRT} and subsequent numerical analyses (e.g.\ \cite{Andreianov2020,AndreianovDiscConstraint2010}) provide complementary tools for stability and approximation.

The list of references is far from exhausted, but we believe that it is enough to get an impression of the actuality and importance of the theme.

\section{Entropy admissible solutions}

In this section, we additionally assume that the flux locally has the representation \eqref{local-flux} and that it satisfies conditions (a), (b), and (c) from the introduction. 

We are not going to work directly with the vanishing viscosity approximation since we cannot control simultaneously the vanishing viscosity term and the flux discontinuity.
The basic idea here is that in the vanishing viscosity approach, the parabolic regularization disappears. Keeping this in mind, we fix the ball $B(\mx_l,r_l)$ and introduce the change of variables $\tilde{\mx}: B(\mx_l,r_l)\to \tilde{\mx}(B(\mx_l,r_l))$ as follows (assuming for definitness that $j=1$)
\begin{equation}
\label{cov}
\tilde{x}_1=x_1-\zeta_l(\hat{x}_1), \ \ (\tilde{x}_2,\dots,\tilde{x}_d)=(x_2,\dots,x_d).    
\end{equation}  In the set $\tilde{\mx}(B(\mx_l,r_l))$, the equation has the form
\begin{equation}
\label{vv-3}
       \begin{split}
    &\pa_t u + \partial_{\tilde{x}_1}\Big(\Big(f^1_L(\tilde{\mx},u)-\sum\limits_{k=2}^d \frac{\pa \zeta_l}{\pa \tilde{x}_k} f^k_L(\tilde{\mx},u)\Big) H(-\tilde{x}_1) \\&\qquad\qquad\qquad  + \Big(f^1_R(\tilde{\mx},u)-\sum\limits_{k=2}^d \frac{\pa \zeta_l}{\pa \tilde{x}_k} f^k_R(\tilde{\mx},u) \Big) H(\tilde{x}_1)\Big) \\& \ \ + \sum\limits_{k=2}^{d}\partial_{\tilde{x}_k}\left( f^k_L(\tilde{\mx},u) H(-\tilde{x}_1) + f^k_R(\tilde{\mx},u)H(\tilde{x}_1)\right)=0.
    \end{split}
\end{equation} Next, we denote $\tilde{\mx}_l=\tilde{\mx}(\mx_l)$ fix a ball $B(\tilde{\mx}_l, R_l)\subset \tilde{\mx}(B(\mx_l,r_l))$ and extend the coefficients radially to the entire space. To be more precise, we set (without changing the notation)
$$
f^k_{L,l}(\tilde{\mx},\lambda)=
\begin{cases}
f^k_{L,l}(\tilde{\mx},\lambda), & \tilde{\mx} \in B(\tilde{\mx}_l,R_l)\\
f^k_{L,l}(R_l\frac{\tilde{\mx}}{|\tilde{\mx}|},\lambda), & \tilde{\mx} \notin B(\tilde{\mx}_l,R_l)\}
\end{cases}
$$ and similarly for all the other coefficients. Then, for a smooth non-negative non-decreasing function $\omega$ such that
\begin{equation}
\label{omega}    
\omega(z)=\begin{cases}
    1, & z>1\\
    0, &z<-1,
\end{cases}
\end{equation} we consider the vanishing viscosity regularization

\begin{equation}
\label{vv-4}
       \begin{split}
    &\pa_t u_\eps + \partial_{\tilde{x}_1}\Big(\Big(f^1_L(\tilde{\mx},u_\eps)-\sum\limits_{k=2}^d \frac{\pa \varphi}{\pa \tilde{x}_k} f^k_L(\tilde{\mx},u_\eps)\Big) \omega(-\tilde{x}_1/\eps) \\&\qquad\qquad\quad  + \Big(f^1_R(\tilde{\mx},u_\eps)-\sum\limits_{k=2}^d \frac{\pa \zeta_l}{\pa \tilde{x}_k} f^k_R(\tilde{\mx},u_\eps) \Big) \omega(\tilde{x}_1/\eps)\Big) \\& \ \ + \sum\limits_{k=2}^{d}\partial_{\tilde{x}_k}\left( f^k_L(\tilde{\mx},u_\eps) \omega(-\tilde{x}_1/\eps) + f^k_R(\tilde{\mx},u_\eps)\omega(\tilde{x}_1/\eps)\right)=\eps \Delta u_\eps.
    \end{split}
\end{equation} By \cite{AM}, we know that the limit of $(u_\eps)$, say $u_l$, satisfies the entropy conditions \eqref{entropy-ineq} with the radially extended coefficients. In particular, in the set $\tilde{\mx}^{-1}(B(\tilde{\mx}_l, R_l))$, the function $u_l$ satisfies \eqref{eqmain}. 

Let us now describe the procedure providing admissibility conditions for the original equation.

\begin{itemize} 

\item[(i)] we fix $\epsilon>0$ such that (b) from the introduction is satisfied and the appropriate neighborhood of $Q_1$ denoted by $Q_1^\epsilon$;

\item[(ii)] for every $\mx_l\in \R^d\setminus Q_1^\epsilon$, we fix a ball $B(\mx_l,r_l)$ in which we have the representation \eqref{local-flux} with $j=1$;

\item[(iii)] we introduce the change of variables \eqref{cov} in the ball $B(\mx_l,r_l)$; 

\item[(iv)] we extend the coefficients of the such obtained equation \eqref{vv-3} radially out the ball $B(\tilde{\mx}_l,R_l)\subset \tilde{\mx}(B(\mx_l,r_l))$ where $\tilde{\mx}_l=\tilde{\mx}(\mx_l)$;

\item[(v)] we consider the vanishing viscosity approximation of \eqref{vv-3} by keeping in mind that we extended the coefficients radially thus preserving its Lipschitz continuity. 

\item[(vi)] we have the equation with the discontinuity hyper-plane $\tilde{x}_1=0$ and we can directly apply the results from \cite{AM} providing existence of entropy admissible solution $u^l$ to \eqref{vv-3};

\item[(vii)] in the ball $B(\tilde{\mx}_l,R_l)$, the constructed solution satisfies \eqref{eqmain} as well as the entropy admissibility conditions \eqref{entropy-ineq};

\item[(viii)] the previous item provide the Kato inequality in that ball or any of its parts for any other function satisfying the same entropy admissibility condition in that ball or its corresponding parts;

\item[(ix)] this assures uniqueness in a cone containing $\tilde{\mx}_l$ implying the fact that any two functions constructed as described (for different  $\mx_l\in \R^d\setminus Q_1^\epsilon$) coincide in the intersection of the corresponding cones;

\item[(x)] since we can cover the space minus $Q_1^\epsilon$ by the basis of the cones and thus cover any relatively compact subset (minus $Q_1^\epsilon$) by a finite number of cones. The minimum of heights of the intersection of any pair of cone, say $t^1_0$, is the time up to which the solution is unique;

\item[(xi)] by repeating the procedure for $Q_j^\epsilon$ we get uniqueness in $(0,t^j_0)$, $j>1$, which implies uniqueness in the interval $(0,\min\limits_{j=1,\dots,d} t_j)$ (since $Q^j_\epsilon$ are disjoint);

\item[(xii)] we repeat the procedure with $u(t_0,\cdot)$ as the initial data, we are able to cover the entire temporal axis.


\end{itemize} Let us now formalize the procedure. 

\begin{theorem}
    Under the non-degeneracy assumption \eqref{non-deg}, conditions (a), (b), and (c),  and local representation of the flux \eqref{local-flux}, there exists an entropy admissible solution $u$ to \eqref{eqmain} in the sense that there exists a bounded function $p_u:\R^+\times\R^d\to \R$, $p_u=u$ almost everywhere with respect to the Lebesgue measure, such that for every $\lambda\in \R$, the following distributional inequality holds: 

\begin{equation}
    \label{entropy-ineq}
    \pa_t |u-\lambda|+\Div \sgn(u-\lambda) \big(\mff(\mx,u)-\mff(\mx,\lambda) \big) + \sgn(p_u-\lambda) \Div \mff(\mx,\lambda) \leq 0.
\end{equation} 
\end{theorem}
\begin{proof}
Let us first denote
\begin{equation*}
\begin{split}
F_L^1(\tilde{\mx},\lambda) &= f^1_L(\tilde{\mx},\lambda)-\sum\limits_{k=2}^d \frac{\pa \zeta_l}{\pa {x}_k} f^k_L(\tilde{\mx},\lambda)\\
F_R^1(\tilde{\mx},\lambda) &= f^1_R(\tilde{\mx},\lambda)-\sum\limits_{k=2}^d \frac{\pa \zeta_l}{\pa {x}_k} f^k_R(\tilde{\mx},\lambda)
\end{split}
\end{equation*} Then, we follow the procedure given in (i)-(xii). This means that for every $\mx_l\in \R^d\setminus Q_{\epsilon/2}$, we can find the function $u^l\in L^\infty(\R^+\times \R^d)$ solving the radially extended equation \eqref{vv-3} and satisfying the entropy admissibility conditions
    \begin{equation}
        \label{ent-adm-transf}
        \begin{split}
        &\pa_t |u^l-\lambda|\\
        &+\pa_{\tilde{x}_1}\Big(\sgn(u^l-\lambda)\big(F_L(\tilde{\mx},u^l)-F_L(\tilde{\mx},\lambda))H(-\tilde{x}_1)+F_R(\tilde{\mx},u^l)-F_R(\tilde{\mx},\lambda))H(\tilde{x}_1) \big) \Big) \\
        &+\sum\limits_{k=2}^{d}\partial_{\tilde{x}_k}\Big(\sgn(u^l-\lambda)\left( (f^k_L(\tilde{\mx},u^l)-f^k_L(\tilde{\mx},\lambda) ) H(-\tilde{x}_1) + (f^k_R(\tilde{\mx},u^l)- f^k_R(\tilde{\mx},\lambda))H(\tilde{x}_1)\right)\\
        &+\sgn(u^l-\lambda)\big(\pa_{\tilde{x}_1}F_L(\tilde{\mx},\lambda)H(-\tilde{x}_1)+\pa_{\tilde{x}_1}F_R(\tilde{\mx},\lambda)H(\tilde{x}_1) \big)\\
        &+\sum\limits_{k=2}^{d} \sgn(u^l-\lambda) \Big(\partial_{\tilde{x}_k} f^k_L(\tilde{\mx},\lambda) H(-\tilde{x}_1) + \partial_{\tilde{x}_k} f^k_R(\tilde{\mx},\lambda)H(\tilde{x}_1)\Big)
        \\&-\sgn(p_u(t,\hat{\tilde{x}}_1)-\lambda) (F_R(\tilde{\mx},\lambda)-F_L(\tilde{\mx},\lambda))\delta(x_1)\leq 0 \ \ {\rm in} \ \ {\cal D}'(\R^+\times \R^d),
        \end{split}
\end{equation} for a bounded function $p_u:\R^+\times \R^d\to \R$.  Keep in mind that we have radially extended functions here. In order to get information on the original equation, we should restrict the support of the test functions employed in \eqref{ent-adm-transf} to the set $\tilde{\mx}(B(\mx_l,r_l))$. Next, we denote
$$
N=N_M(R)=\max\limits_{\substack{\mx \in B(0,R)\\ |\lambda|\leq M}} \Big[ \sum\limits_{j=1}^d (|\pa_\lambda f_L^j(\mx,\lambda)|^2+|\pa_\lambda f_R^j(\mx,\lambda)|^2 ) \Big]^{1/2}
$$ and consider the cones
$$
{\cal K}(\mx_l, R_l)=\{(t,\mx): \, |{\mx}-\tilde{\mx}_l|<R-Nt, \ \ t\in [0,RN^{-1})\}, \ \ \tilde{\mx}_l=\tilde{\mx}(\mx_l)
$$ for $R_l$ small enough such that
$$
{\cal S}_0(\mx_l, R_l)\subset \tilde{\mx}(B(\mx_l,r_l)) 
$$ where we denoted by ${\cal S}_\tau(\mx_l, R_l)$ the cross-section of the cone with the plane $t=\tau$ (i.e. it is a hyper-ball centered at $\tilde{\mx}_0)$ with the radius $R_l$). 

For a fixed relatively compact set $K\subset\subset \R^d$, the cones basis ${\cal S}_0(\mx_l, R_l)$, $\mx_l\in \bar{K}\setminus Q^1_{\epsilon/2}$ cover entire $K\setminus Q^1_{\epsilon/2}$. Since $\bar{K}\setminus Q^1_{\epsilon/2}$ is compact, we can extract its finite covering which corresponds to the sequence of cones
\begin{equation}
\label{finite}    
{\cal K}(\mx_l, R_l), \ \ l=1,\dots,n.
\end{equation}  and denote by $t_1$ the maximal height of the cross-sections such that
$$
 [0,t_1)\times K \setminus Q_{\epsilon} \subset \cup_{j=1}^n \{(t,\mx): \, |{\mx}-\tilde{\mx}_l|<R-Nt, \ \ t\in [0,t_1)\}.
$$ The time $t_1$ is actually the minimum between the heights at which two cones with non-empty intersection intersect and maximum of the heights at which cones intersect the cylinder $Q_{\eps}$. To this end, let us consider the intersection of two fixed cones whose intersection is non-empty. Let us denote them by ${\cal K}(\mx_i, R_i)$ and ${\cal K}(\mx_j, R_j)$ and the corresponding functions satisfying \eqref{ent-adm-transf} at ${\cal K}(\mx_i, R_i)$ and ${\cal K}(\mx_j, R_j)$, we denote by $u^i$ and $u^j$, respectively. In the intersection
$$
{\cal K}(\mx_i, R_i) \cap {\cal K}(\mx_j, R_j), 
$$ we have the Kato inequality satisfied
\begin{equation}
    \label{kato-int}
    \begin{split}
        &\pa_t |u^i-u^j|\\
        &+\pa_{\tilde{x}_1}\Big(\sgn(u^i-u^j)\big(F_L(\tilde{\mx},u^i)-F_L(\tilde{\mx},u^j))H(-\tilde{x}_1)+F_R(\tilde{\mx},u^i)-F_R(\tilde{\mx},u^j))H(\tilde{x}_1) \big) \Big) \\
        &+\sum\limits_{k=2}^{d}\partial_{\tilde{x}_k}\Big(\sgn(u^i-u^j)\left( (f^k_L(\tilde{\mx},u^i)-f^k_L(\tilde{\mx},u^j) ) H(-\tilde{x}_1) + (f^k_R(\tilde{\mx},u^i)- f^k_R(\tilde{\mx},u^j))H(\tilde{x}_1)\right)\\
        &+\sgn(u^i-u^j) \big(\pa_{\tilde{x}_1}F_L(\tilde{\mx},\lambda)|_{\lambda=u^i}-\pa_{\tilde{x}_1}F_L(\tilde{\mx},\lambda)|_{\lambda=u^j} \big)H(-\tilde{x}_1)\\& +\sgn(u^i-u^j) \big( \pa_{\tilde{x}_1}F_R(\tilde{\mx},\lambda)|_{\lambda=u^i} -\pa_{\tilde{x}_1}F_R(\tilde{\mx},\lambda)|_{\lambda=u^j} \big) H(\tilde{x}_1)\\
        &+\sum\limits_{k=2}^{d} \sgn(u^i-u^j) \big(\partial_{\tilde{x}_k} f^k_L(\tilde{\mx},\lambda)|_{\lambda=u^i}-\partial_{\tilde{x}_k} f^k_L(\tilde{\mx},\lambda)|_{\lambda=u^j}\big) H(-\tilde{x}_1)
        \\& + \sum\limits_{k=2}^{d} \sgn(u^i-u^j) \big( \partial_{\tilde{x}_k} f^k_R(\tilde{\mx},\lambda)|_{\lambda=u^i} - \partial_{\tilde{x}_k} f^k_R(\tilde{\mx},\lambda)|_{\lambda=u^j} \big) H(\tilde{x}_1)       
        \leq 0,
        \end{split}
\end{equation} in ${\cal D}'({\cal K}(\mx_i, R_i) \cap {\cal K}(\mx_j, R_j))$.

Indeed, if the cones intersection does not contain the discontinuity line, then the conclusion follows from the standard Kruzhkov procedure \cite{Kru}. If the continuity line is present, then the conclusion follows from \cite{AM} and the fact that, by the assumption (c) in the introduction, the representation of continuity line coincides in the neighborhoods $B(\mx_l,r_l)$ whose intersection contains the discontinuity line. 

Let us now derive that $u^i=u^j$ almost everywhere in ${\cal K}(\mx_i, R_i) \cap {\cal K}(\mx_j, R_j)$. We are adapting the standard Kruzhkov procedure from \cite{Kru}. For the function $\omega$ from \eqref{omega}, we denote
$$
\chi_\eps(t,\tilde{\mx})=1-\omega\left(\frac{|\tilde{\mx}-\tilde{\mx}_i|+Nt-R_i+\eps}{\eps}\right)\omega\left(\frac{|\tilde{\mx}-\tilde{\mx}_j|+Nt-R_j+\eps}{\eps}\right) 
$$and note that $\chi(t,\tilde{\mx})\equiv 0$ outside the intersection of the cones ${\cal K}(\mx_i, R_i) \cap {\cal K}(\mx_j, R_j)$. Moreover, we have
\begin{equation*}
\begin{split}
    &0=\pa_t \chi_\eps +\sum\limits_{k=1}^d N |\omega\left(\frac{|\tilde{\mx}-\tilde{\mx}_i|+Nt-R_i+\eps}{\eps}\right) \pa_{\tilde{x}_k} \omega\left(\frac{|\tilde{\mx}-\tilde{\mx}_j|+Nt-R_j+\eps}{\eps}\right)|\\&\qquad\quad+\sum\limits_{k=1}^d N|\omega\left(\frac{|\tilde{\mx}-\tilde{\mx}_j|+Nt-R_j+\eps}{\eps}\right) \pa_{\tilde{x}_k} \omega\left(\frac{|\tilde{\mx}-\tilde{\mx}_i|+Nt-R_i+\eps}{\eps}\right)|\\& \geq \pa_t \chi_\eps + \sum\limits_{k=1}^d \Big( \frac{f_L^k(\tilde{\mx},u_i)-f_L^k(\tilde{\mx},u_j)}{u_i-u_j}H(-\tilde{x}_1) +\frac{f_R^k(\tilde{\mx},u_i)-f_R^k(\tilde{\mx},u_j)}{u_i-u_j} H(\tilde{x}_1)\Big)\pa_{\tilde{x}_k} \chi_\eps.
\end{split}
\end{equation*} Keeping this in mind and inserting the test function
$$
\big(\omega((t-\rho)/\eps)-\omega((t-\tau)/\eps)\big)\chi(t,\tilde{\mx}) 
$$ into \eqref{kato-int}, we conclude (we denote ${\cal K}_i={\cal K}(\mx_i, R_i)$ and ${\cal K}_j={\cal K}(\mx_j, R_j)$)
\begin{equation*}
    \begin{split}
        &\int_0^T\int_{{\cal K}_i \cap {\cal K}_j} \Big(\frac{1}{\eps}\omega'\big(\frac{t-\rho}{\eps} \big)- \frac{1}{\eps}\omega'\big(\frac{t-\tau}{\eps} \big) \Big) \chi_\eps(t,\tilde{\mx}) |u^i-u^j| \, d\tilde{\mx} dt 
        \\& \leq 2d C\int_0^T\int_{\R^d} \big(\omega((t-\rho)/\eps)-\omega((t-\tau)/\eps)\big)\chi(t,\tilde{\mx}) |u^i-u^j| \, d\tilde{\mx} dt  
    \end{split}
\end{equation*} where 
$$
C=\max\limits_{\substack{\mx \in B(0,R)\\ |\lambda|\leq M}} \Big[ \sum\limits_{j=1}^d (|\pa^2_{x_j\lambda} f_L^j(\mx,\lambda)|^2+|\pa^2_{x_j\lambda} f_R^j(\mx,\lambda)|^2 ) \Big]^{1/2}.
$$ After letting $\tau\to 0$ here (keep in mind that we have the strong trace of $u^i$ and $u^j$ at every level $t=t_0$ and at $t_0=0$ as well \cite{crasta2, pan_traces}) and using the Gronwal lemma, we conclude that
$$
u^i(t,\tilde{\mx})=u^j(t,\tilde{\mx}) \text{  for a.e.  } (t,\tilde{\mx})\in  {\cal K}_i \cap {\cal K}_j.
$$ From here, we see that we have a function $u$ which coincides with $u^l$ on every ${\cal K}(\mx_l,R_l)$, $l=1,\dots,n$, from \eqref{finite} and satisfies the entropy conditions \eqref{ent-adm-transf} (with the original coefficients) for $t<t_{01}$ for $t_{01}$ small enough, and outside $Q_1^\epsilon\times (0,t_{01})$. Moreover, such a function is unique.

Now, we repeat the procedure starting with the set $Q^2_{\epsilon/2}$ and considering the local representation of the discontinuity manifold of the form $x_2=\zeta_l(\hat{x}_2)$. We shall obtain the function $u$ satisfying the entropy conditions \eqref{ent-adm-transf} (with the original coefficients) for $t<t_{02}$ for $t_{02}$ small enough, and outside $Q_2^\epsilon\times (0,t_{02})$. 

We note that the functions constructed in this two steps coincide on $K\setminus (Q_1^\epsilon\cup Q_2^\epsilon)$ and thus can be extended to $Q_1^\epsilon$ and $Q_2^\epsilon$ using the functions constructed in the second and first steps, respectively, so that they satisfy the entropy admissibility conditions.

Thus, we have obtained the unique solution on $K\setminus (Q_1^\epsilon\cap Q_2^\epsilon)$ on $[0,\min\{t_{01},t_{02}\})$ since the function $u$ constructed in the first part out of $Q_1^\epsilon$ covers $Q_2^\epsilon\setminus Q_1^\epsilon$ and the function constructed out of $Q_2^\epsilon$ covers $Q_1^\epsilon\setminus Q_2^\epsilon$. Continuing the procedure, we reach to the function $u$ satisfying \eqref{ent-adm-transf} on $K\setminus (\cap_{k=1}^d Q_k^\epsilon)=K$ on $[0,t_0)$ where $t_0=\min\limits_{k=1,\dots,d} t_{0k}$. Note that here we used the assumption (b) from the introduction.

Finally, repeating the procedure with the initial data $u(t_0,\mx)$ we conclude that there exists $u$ satisfying locally \eqref{ent-adm-transf}. By returning to the original variables $(t,\mx)$, we conclude the theorem. \end{proof}

The procedure from the previous theorem together with the results from \cite{AM} provides uniqueness of the entropy admissible solution. However, we conclude the result directly from  \cite{crasta1}.

\begin{corollary}
    Under the non-degeneracy assumptions \eqref{non-deg} and the assumptions (a), (b), and (c) from the introduction, there exists a unique solution to the Cauchy problem \eqref{eqmain} which satisfies entropy admissibility conditions \eqref{entropy-ineq}.
\end{corollary}

\section{$(\eps_k)$-vanishing viscosity germ}

We start by fixing a smooth $\mx$-approximation $\mff_\eps(\mx,\lambda)$ of the flux $\mff(\mx,\lambda)$ such that it still holds
$$
\mff_\eps(\mx,a)=\mff_\eps(\mx,b)=0.
$$Then, from the results in \cite{panov_arma}, thanks to the non-degeneracy conditions, we know that the vanishing viscosity approximation to \eqref{eqmain} admits $L^1_{loc}(\R^+\times \R^d)$-strongly converging subsequence. In other words, the family of solutions to

\begin{equation}
\label{eqmain-vv}
\begin{split}
\partial_t u_\eps &+ \Div \mff_\eps(\mx,u_\eps) = \eps \Delta u_\eps, \ \ \\
u\big|_{t=0}&= u_0(\mx) \in L^1(\R^d), \ \ a \leq u_0(\mx) \leq b, \ \ (t,\mx)\in \R^+\times \R^d,
\end{split}
\end{equation} admits a subsequence $(u_{\eps_k})$ strongly converging in $L^1_{loc}(\R^+\times \R^d)$-strongly converging subsequence.

In this section, we consider \eqref{eqmain} under the mere assumption that the flux belongs to $BV(\R^d;C^1(\R))$. We shall construct a complete germ consisting of limits of converging subsequences to the vanishing viscosity approximation \eqref{eqmain-vv}. To this end, with respect to the $L^1(\R^d)$-topology, we fix a countable dense set ${\cal C} \subset L^1(\R^d)\cap L^\infty(\R^d)$ of functions ranging between $a$ and $b$:
\begin{equation}
\label{cnt-set}
{\cal C}=\{ u^l_0\in L^1(K)\cap L^\infty(\R^d): \, a \leq u_0^k(\mx) \leq b, \ \ l\in \N, \, \mx\in \R^d\}.
\end{equation} For each $u|_{t=0}=u_0^l$, we solve the Cauchy problem \eqref{eqmain-vv} and get the family of solutions $(u_\eps^l)$. Since we assumed that the flux $\mff$ satisfies the non-degeneracy conditions \eqref{non-deg}, by \cite{panov_arma}, we can find a zero sequence $(\eps_k)$ such that for every $l\in \N$, the sequence $(u_{\eps_k}^l)$ converges in $L^1(K)$ toward a function $u_l\in L^1(K) \cap L^\infty(\R^d)$. Keeping this in mid, we shall look for the following stable class of solutions to \eqref{eqmain}.

\begin{definition}
    For the fixed zero sequence $(\eps_k)$, we say that the $(\eps_k)$-vanishing viscosity germ is the set ${\cal G}_{(\eps_k)}$ consisting of the functions $u\in L_{loc}^1(\R^d)\cap L^\infty(\R^d)$ such that for every $u_0\in L^1(\R^d)\cap L^\infty(\R^d)$, $a\leq u_0(\mx) \leq b$, $\mx \in \R^d$, there exists $u\in {\cal G}_{(\eps_k)}$ representing the weak solution to \eqref{eqmain} and obtained as the limit along the fixed subsequence $(\eps_k)$ to the vanishing approximation \eqref{eqmain-vv} with $\eps$ replaced by $\eps_k$.   
\end{definition}

From the latter definition, the following theorem is fairly straightforward corollary of the uniform stability of solutions to the parabolic approximations.

\begin{theorem}
    There exists a zero sequence $(\eps_k)$ The germ ${\cal G}_{(\eps_k)}$ is complete i.e. for every $u_0\in L^1(\R^d)\cap L^\infty(\R^d)$, $a\leq u_0(\mx) \leq b$, $\mx \in \R^d$, there exists $u\in {\cal G}_{(\eps_k)}$ satisfying \eqref{eqmain-vv} in the weak sense. Moreover, solutions from the germ ${\cal G}_{(\eps_k)}$ are stable.
\end{theorem}
\begin{proof}
    Let us first show completness of the germ. To this end, we fix $u_0\in L^1(\R^d)\cap L^\infty(\R^d)$, $a\leq u_0(\mx) \leq b$, $\mx \in \R^d$, and the countable set ${\cal C}$ from \eqref{cnt-set}. The construction from the beginning of the section provides the zero sequence $(\eps_k)$ such that for every $u_0\in {\cal C}$, we have $u\in {\cal G}_{(\eps_k)}$.

    Let us now assume that $u_0\notin {\cal C}$. We take the sequence $(u_0^l)$, $u_0^l\in {\cal C}$, such that
    $$
    \lim\limits_{l\to \infty} \| u_0^l-u_0 \|_{L^1(\R^d)}=0.
    $$ We then consider the parabolic problems \eqref{eqmain-vv} with the initial data $u_0$ and $u_0^l$ with the solutions $u_{\eps_k}$ and $u_{\eps_k}^l$, respectively. By subtracting the two equations, we get
    $$
    \pa_t (u_{\eps_k}-u^l_{\eps_k})+\Div \big(\mff_{\eps}(\mx,u_{\eps_k})-\mff_{\eps}(\mx,u^l_{\eps_k}\big) =\eps_k \Delta (u_{\eps_k}-u^l_{\eps_k}). 
    $$ We multiply the equation by $\sgn(u_{\eps_k}-u^l_{\eps_k})$ and after standard manipulations, taking into account the Leibnitz rule and the fact that $(\big(\mff_{\eps}(\mx,u)-\mff_{\eps}(\mx,v)\big))\delta(u-v)=0$, we get in the sense of distributions
    \begin{equation}
    \label{kato-vv}
        \pa_t |u_{\eps_k}-u^l_{\eps_k}|+\Div \big( \sgn(u_{\eps_k}-u^l_{\eps_k}) \, \big(\mff_{\eps}(\mx,u_{\eps_k})-\mff_{\eps}(\mx,u^l_{\eps_k}\big) \big) \leq \eps_k \Delta |u_{\eps_k}-u^l_{\eps_k}|.
    \end{equation} We integrate give expression over $(0,t)\times \R^d$ and let $\eps_k\to 0$. We have by denoting $u^l$ the limit of $u_{\eps_k}^l$
    \begin{equation*}
        \lim\limits_{\eps_k\to 0}\int_{\R^d}|u_{\eps_k}(t,\mx)-u^l(t,\mx)|\, d\mx \leq \lim\limits_{\eps_k\to 0}\int_{\R^d} |u_{0}(\mx)-u_0^l(\mx)|\, d\mx. 
    \end{equation*} By letting $l\to \infty$ here, we see that $(u_{\eps_k})$ will have the $L^1(\R^d)$-limit if $(u^l)$ has the $L^1(\R^d)$-limit. However, it is easy to see that the sequence $(u^l)$ is the Cauchy sequence in $L^1(\R^d)$. Indeed, we can replace $u_{\eps_k}$ and $u_{\eps_k}^l$ by $u_{\eps_k}^{l_1}$ and $u_{\eps_k}^{l_2}$ in \eqref{kato-vv} and by integration over $(0,t)\times \R^d$ and letting $\eps_k\to 0$, we get
    \begin{equation}
    \label{stability-vv}
        \int_{\R^d}|u^{l_1}(t,\mx)-u^{l_2}(t,\mx)|\, d\mx \leq \int_{\R^d} |u^{l_1}_{0}(\mx)-u_0^{l_2}(\mx)|\, d\mx. 
    \end{equation} Since $(u_0^{l})$ is by construction convergent and thus Cauchy sequence, the proof is concluded.

    We finally note that the stability follows by replacing $u^{l_1}$ and $u^{l_2}$ by arbitrary $u,v \in {\cal G}_{(\eps_k)}$. 
\end{proof} A natural question here is how many different stable germs arising from the vanishing viscosity approximation we have. In one dimensional case when the discontinuity manifold is fixed, we know from \cite[Sect. 5]{AKR} that there exists a unique such a germ i.e. that entire vanishing viscosity approximation converges (provided non-degeneracy conditions). However, our assumptions seems to be too weak to ensure such a strong result. To this end, we strengthen our assumption in the next section in order to obtain an approximation independent description of admissible solutions.

\section*{Acknowledgements}\label{sec:thanks}
This work was supported in part by the Stand Alone project P 35508 of the Austrian Science Fund FWF.

\section*{Declarations}
\subsection*{Data Availability} Data sharing not applicable to this article as no datasets were generated or analysed during
the current study.

\subsection*{Conflict of interests}
The authors have no competing interests to declare that are relevant to the content of
this article.

\end{document}